\documentclass[10pt,a4paper]{article}
\usepackage{amssymb,amsmath,amsfonts,amsthm,authblk,hyperref,here,enumerate,
xfrac,mathtools,graphicx,pstool,xcolor,pgf,tikz,array,fancyhdr}
\usepackage[all]{xy}
\usetikzlibrary{arrows,arrows.meta}
\usepackage[english]{babel}
\usepackage{titlesec}
\usepackage{lato,sfmath,microtype}
\usepackage[utf8]{inputenc}
\numberwithin{equation}{section}
\graphicspath{{figs/}}

\newcommand{\R}{\mathbb R}
\newcommand{\N}{\mathbb N}
\newcommand{\Z}{\mathbb Z}
\newcommand{\C}{\mathbb C}
\newcommand{\Rangle}{\rangle\!\rangle}
\newcommand{\Langle}{\langle\!\langle}
\definecolor{darkgreen}{rgb}{0,.6,0}

\long\def\@makefnmark{%
        \hbox {\@textsuperscript {\sffamily\@thefnmark}}
        }

\usepackage{caption}
\titleformat{\section}
  {\normalfont\sffamily\Large\bfseries}
  {\thesection}{1em}{}
  \titleformat{\subsection}
  {\normalfont\sffamily\bfseries\large}
  {\thesubsection}{1em}{}
    \titleformat{\subsubsection}
  {\normalfont\sffamily\bfseries\normalsize}
  {\thesubsubsection}{1em}{}
\definecolor{darkercyan}{RGB}{0,169,170}
\definecolor{MyDarkBlue}{rgb}{0,0.29,0.7}
\hypersetup{
    colorlinks=true,       
    linkcolor=darkercyan,          
    linkbordercolor=darkercyan,
    citecolor=darkercyan,
    citebordercolor=darkercyan,        
    filecolor=darkercyan,      
    urlcolor=darkercyan           
}

\newtheoremstyle{plain}
  {10pt}
  {10pt}
  {\it \sffamily}
  {0pt}
  {\bf \sffamily}
  {}
  {\newline}
  {}
\newtheoremstyle{definition}
  {10pt}
  {10pt}
  {\sffamily}
  {0pt}
  {\bf \sffamily}
  {}
  {\newline}
  {}
  
\renewenvironment{proof}{{\textit{\sffamily Proof.}}}{\hfill\qedsymbol}
\theoremstyle{plain}

\usepackage{tocloft}

\newtheorem{theorem}{Theorem}[section]

\newtheorem{lemma}[theorem]{Lemma}
\newtheorem{proposition}[theorem]{Proposition}
\newtheorem*{theorem*}{Theorem}
\newtheorem*{lemma*}{Lemma}
\theoremstyle{definition}

\newtheorem*{remark}{Remark}

\usepackage{xpatch}
\makeatletter
\xpatchcmd{\@makefnstartbox}{%
\footnotesize}%
{\footnotesize\sffamily}{}{}
\xpatchcmd{\@makefnmark}{\normalfont}{\sffamily}{}{}
\makeatother
  
\fancypagestyle{plain}{%
\fancyhf{} 
\fancyfoot[C]{\sffamily\fontsize{10pt}{10pt}\sffamily\thepage} 

}
\begin{document}

\title{\sffamily\textbf{Application of a Fourier-Type Series Approach based on Triangles of Constant Width to Letterforms}}
\author[1]{\sffamily Micha Wasem}
\author[2]{\sffamily Florence Yerly}
\affil[1,2]{\normalsize \sffamily School of Engineering and Architecture, HES-SO University of Applied Sciences and Arts Western Switzerland, P\'erolles 80, 1700 Fribourg, Switzerland\\{\tt micha.wasem@hefr.ch}, {\tt florence.yerly@hefr.ch}}
\affil[1]{\normalsize \sffamily Faculty of Mathematics and Computer Science, UniDistance Suisse, Schinerstrasse 18, 3900 Brig, Switzerland\\{\tt micha.wasem@fernuni.ch}}
\date{\sffamily \today}
\maketitle
\sffamily
\begin{center}
\includegraphics[scale=0.58]{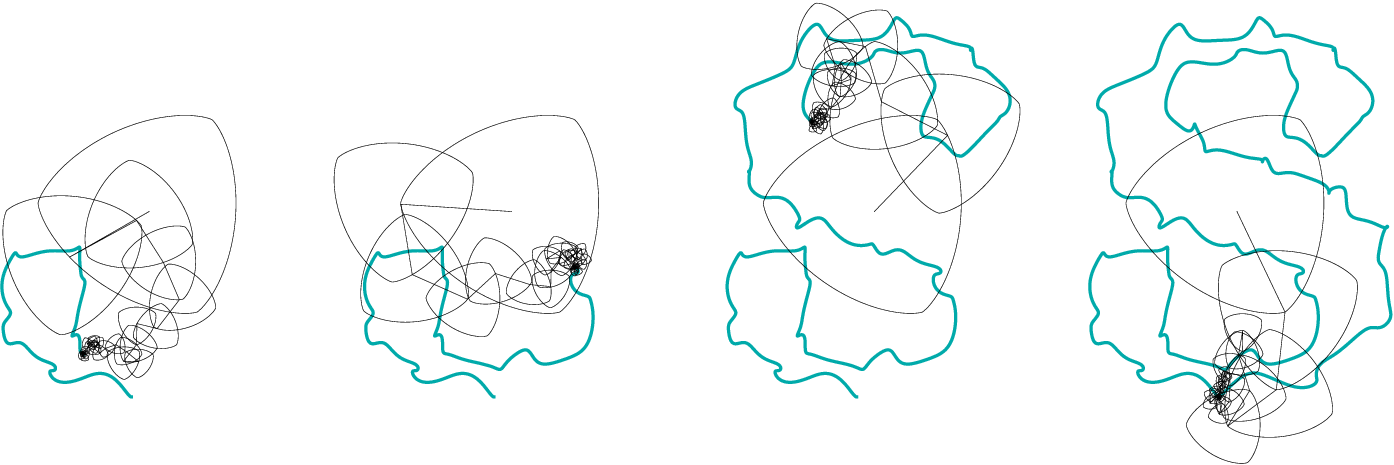}
\end{center}
\begin{abstract}
In this work, we present a novel approach to type design by using Fourier-type series to generate letterforms. 
We construct a Fourier-type series for functions in $L^2(S^1,\C)$ based on triangles of constant width instead of circles to model the curves and shapes that define individual characters. In order to compute the coefficients of the series, we construct an isomorphism $\mathcal R:L^2(S^1,\C)\to L^2(S^1,\C)$ and study its application to letterforms, thus presenting an alternative to the common use of B\'ezier curves.
The proposed method demonstrates potential for creative experimentation in modern type design.
\end{abstract}
\section{\sffamily Introduction}

Modern type design is largely shaped by the smooth behavior of Bézier curves. In this study, we introduce an alternative approach based on functional analysis that models letterform contours as periodic plane curves and examines their structure using Fourier analysis. Instead of representing contours as sums of circles, we employ triangles of constant width, a modification that can produce either smooth, symmetric contours or, conversely, deliberately irregular and rough ones.

In order to make this more precise, let $\mathrm L\in L^2(S^1,\C)$ be a (suitably regular\footnote{\sffamily For the needs of our paper, \emph{Hölder-continuity} is enough.}) parametrization of  a letterform contour. Then the Fourier series of $\mathrm L$ represents the contour as a uniformly convergent series of circles by a theorem of \cite{c_theory_1931}:
$$
\mathrm L(t) = \sum_{k\in \Z}  c_k \mathrm e^{\mathrm ikt}
$$
Truncating the above series leads to an approximation
\begin{equation}\label{truncatedfourierseries}
\mathrm L_m(t)=\sum_{k\in\Z \atop |k|\leqslant m} c_k \mathrm e^{\mathrm ikt}.
\end{equation}
\begin{figure}
\begin{center}
\includegraphics[scale=0.3]{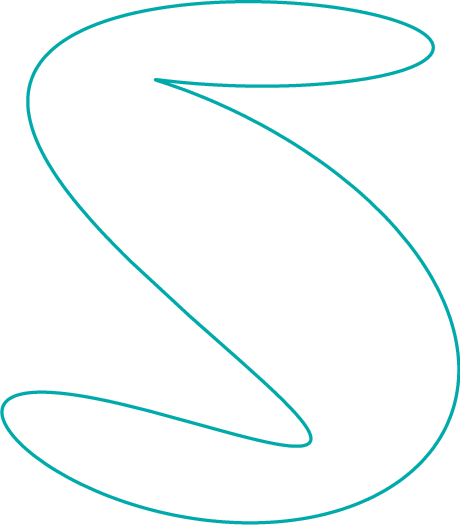}\qquad\includegraphics[scale=0.3]{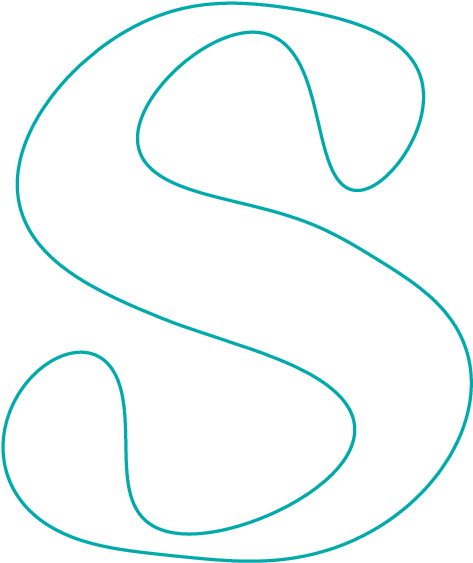}\qquad
\includegraphics[scale=0.3]{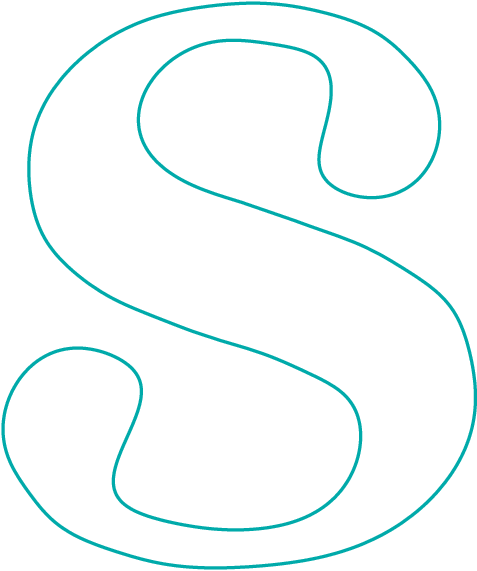}\qquad\includegraphics[scale=0.3]{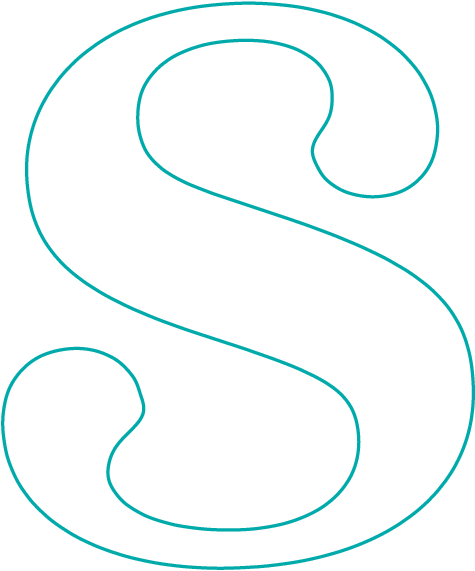}\\[3mm]
\includegraphics[scale=0.3]{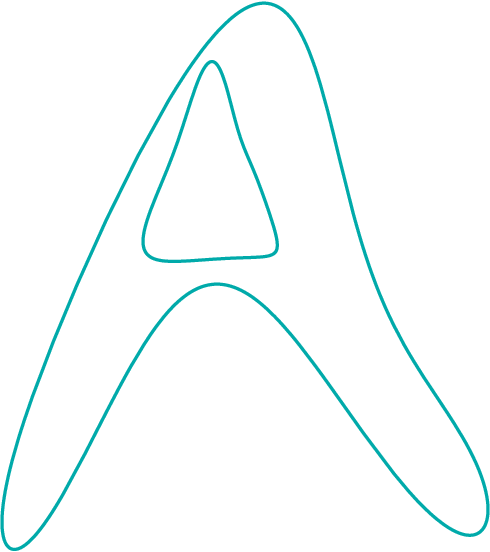}\qquad
\includegraphics[scale=0.4]{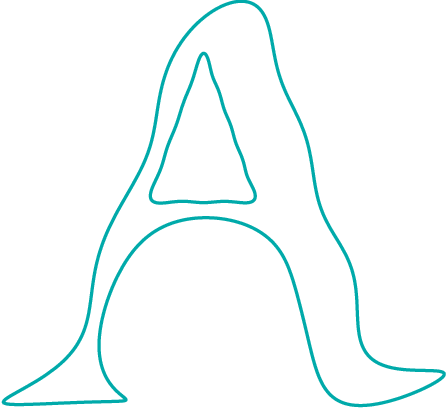}\qquad
\includegraphics[scale=0.4]{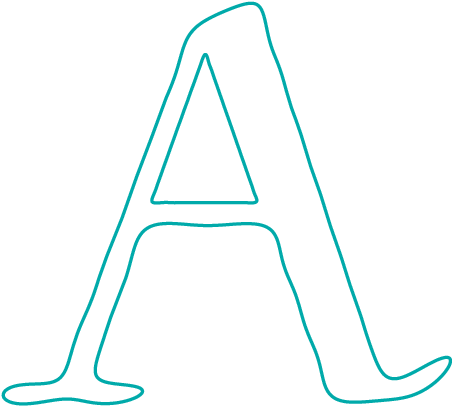}\qquad
\includegraphics[scale=0.4]{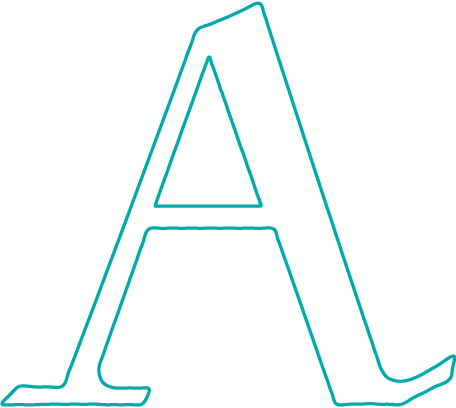}
\caption{Letters ``S'' from Romain 20 and ``A'' from Mitim Sigma obtained from a truncated Fourier series for the choices $m=4,10,25$ and $m=100$ as given by equation \eqref{truncatedfourierseries}.}\label{alice-fourier}
\end{center}
\end{figure}As an example, we have computed the Fourier series of a letter ``S'' from the Romain 20 font 
designed by Alice Savoie \cite{Romain20} and the one of a letter ``A'' from the font Mitim Sigma designed by Radim Pe\v sko \cite{MitimSigma} and visualized their truncated series (see Figure \ref{alice-fourier}).
It is worth noting that Fourier analysis has countless applications in a wide variety of fields and has already proven its value in the creation of mathematical art, see e.g. \cite{Farris2013}. In our context however, the truncated Fourier series provide us with a first attempt for transforming a letterform, but they offer a fairly limited artistic choice. First, one can choose the number of terms in the series, knowing that the larger the number, the more the modification will resemble the original letter as illustrated in Figure \ref{alice-fourier}. If the number of terms is small, the transformation will produce a letter resembling an inflated balloon. 
Secondly, one can modify a coefficient $c_k$ in equation \eqref{truncatedfourierseries}, which makes the result difficult to predict due to the global nature of the Fourier transform.

In this paper, we will instead consider a series representation
\begin{equation}\label{constantwidthseries}
\mathrm L(t) = \sum_{k\in\Z} \widetilde{c_k} \gamma(k t), 
\end{equation}
where $\gamma(t) = \mathrm e^{\mathrm i t}\left(1+a\cos(3t)-3\mathrm i a \sin(3t)\right)$ is a parametrization of a triangle of constant width as long as $a\in[0,\frac18]$ (see item (\ref{item2}) in the collection of basic facts about support functions), and where $a$ is a tuning parameter (see Figure \ref{fig:trianglesofconstantwidth}).
\begin{figure}
\begin{center}
\includegraphics[scale=0.88]{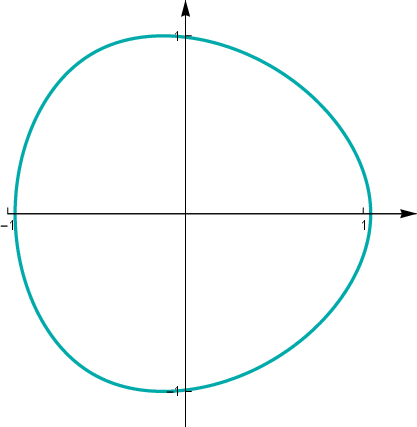}\qquad\includegraphics[scale=0.88]{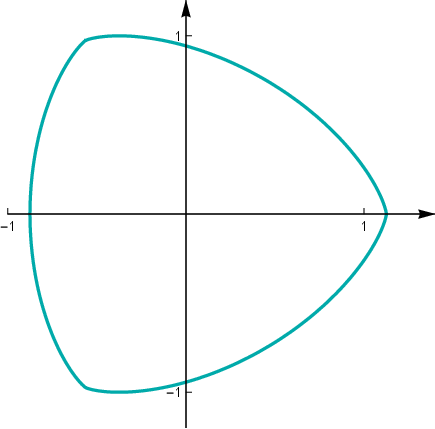}
\caption{The images of $\gamma$ for the choices $a=\frac{1}{24}$ and $a=\frac18$ as subsets of $\R^2\cong\C$ are triangles of constant width, which serve as a basis of a Fourier-type series expansion.}\label{fig:trianglesofconstantwidth}
\end{center}
\end{figure}
Observe that \eqref{constantwidthseries} is a generalization of the standard Fourier series since it coincides with it whenever $a=0$.
The choice of working with $\gamma$ is motivated as follows: our primary goal was to generalize the classical Fourier series by a corresponding series with another basic building block that generalizes the circle. For this -- as seen for instance in Wankel engines -- bodies with constant width serve as a natural choice to do so as they share the property of having constant width with the circle. If one takes $n$-gons with constant width for $n$ odd, they tend to become similar to a circle as $n$ becomes large and their symmetry groups become larger as well. In that sense, $n=3$ is a choice which differs most in terms of breaking symmetry from the circle and, as we will show, produces the most notable differences in the letterforms. 

Furthermore, as it turns out, the coefficients $\widetilde{c_k}$ in \eqref{constantwidthseries} are obtained as the standard Fourier coefficients of $\mathcal R(\mathrm L)$, where $\mathcal R$ is an isomorphism of $L^2(S^1,\C)$. Curiously, the isomorphism $\mathcal R$ restricts to a map $C^k(S^1,\C)\to C^k(S^1,\C)$ provided $a<{1}/(2^k(2^k+2))$ (see Theorem \ref{C^k-restriction}) and the result is sharp in the sense that if $a\geqslant 1/(2^k(2^k+2))$, then the image of a $C^k$ function under $\mathcal R$ will not be $C^k$ in general. Therefore, the tuning parameter $a$ can be understood as a regularity parameter which can be used to obtain ``quirks'' and textures in the letterforms, offering a formal method to generate everything from perfectly smooth symmetries to intentionally jagged and irregular contours. As established in Section \ref{type-design}, the series given by \eqref{constantwidthseries} even converges in $L^2(S^1,\C)$, if $\frac18<a<\frac13$, although in that case, $\gamma$ is a curve with self-intersections. In that case, the transformation of a letterform will result in a fractal-like letterform with self-intersections (see Figure \ref{1/5}).

To provide additional context, we summarize the prevailing methodology in modern type design, which is fundamentally local in nature and relies on Bézier curves. 
These curves -- popularized by Pierre Bézier but previously investigated by Pierre de Casteljau
\cite{bezier_mathematical_1986,bezier_essai_1977,hosaka_bezier_1992,van_to_development_1992,maqsood_novel_2020,fitter_review_2014,de_casteljau_outillages_1959}
-- are linear combinations of Bernstein polynomials \cite{bernstein_proof_nodate}. A planar Bézier curve of degree $n$ is uniquely determined by $n+1$ control points $p_0,\dots ,p_n\in \mathbb{R^{\mathnormal{2}}}$ and is given by

$$
t\mapsto \sum_{k=0}^n \binom nk t^k(1-t)^{n-k}p_k,  \qquad t\in[0,1].
$$

Because they are easy to scale, rotate, and store, Bézier curves possess several computational advantages. Although they cannot represent all shapes exactly -- for instance, a perfect circle cannot be realized -- they are dense in the space of continuous curves, so any shape can be uniformly approximated by a suitable collection of Bézier segments \cite{natanson_constructive_1964-1}.

The use of Bézier curves has therefore become standard in digital type design. TrueType fonts employ piecewise quadratic Bézier curves, whereas systems such as Metafont rely on piecewise cubic ones \cite{bigelow_font_2020}. This piecewise construction highlights the inherently local character of the method: each segment is controlled by a small number of points, and designers typically adjust curves by manipulating these local control structures.

A common criticism among type designers is that the use of B\'ezier curves leads typesetters to create similar fonts and symbols which are too clean, without irregularities or quirks. Indeed, some typeface softwares suggest improvements to the type designer -- for instance, FontLab \cite{Fontlab}
and Glyphs \cite{Glyphs}

have features which highlight certain control points of the B\'ezier curves where the curvature is very high or discontinuous, or lines that are not perfectly horizontal or vertical. This highlighting encourages users to correct their designs, thereby urging to standardize fonts\footnote{\sffamily Personal communication with Raphaela Häfliger, Alice Savoie, Kai Bernau, Nicolas Bernklau, Matthieu Cortat, Roland Früh and Radim Pe\v{s}ko, October 2023 at ECAL/Ecole cantonale d'art de Lausanne} (see also \cite{bbwebsite}).

In contrast to local, segment‑based techniques, our method adopts a global perspective by modeling entire letterform contours as periodic plane curves and analysing them through Fourier analysis, a direction that, to our knowledge, has not been explored in the existing literature. This global representation provides coherent control over the shape as a whole and makes it possible to introduce calibrated deviations from smoothness, yielding controlled roughness that is difficult to obtain with traditional Bézier‑based tools. Within this framework, transformations arise from Fourier‑type series in which circles are replaced by triangles of constant width.

Section \ref{section:2} provides the needed mathematical background and Section \ref{type-design} shows how the material can be applied to type design. In order to facilitate the reading, all mathematical proofs have been collected in the appendix.

\section{\textcolor{black}{Mathematical foundations}}\label{section:2}

\textcolor{black}{We develop here the mathematical tools that will enable us to construct the Fourier-based transformation by replacing the circle with triangles of constant width. With the Fourier series, we approximate a closed curve by a sum of circles moving around other circles. Each circle has its own radius and frequency. These circles are easy to parametrize with $t\mapsto c_k \mathrm e^{\mathrm{i}kt}$. The first step for our transformation is to compute a parametrization of the triangle of constant width and for that we will need support functions. We can then build our ``constant-width-triangle-Fourier-based'' transform and compute the inverse transformation and an inner product, giving us all the ingredients we need to calculate the transformation of a letterform contour. 
 }
\textcolor{black}{\subsection{Recall of some mathematical concepts}}
We will use the identification $S^1 \cong \mathbb R/2\pi\mathbb Z$ throughout this article and the standard inner product of $L^2(S^1,\mathbb C)$ will be given by
$$
\langle f,g\rangle = \frac{1}{2\pi}\int_0^{2\pi}\!\!f(t)\overline{g(t)}\,\mathrm dt
$$ so that $\{t\mapsto \mathrm e^{\mathrm ikt}\}_{k\in\mathbb Z}$ is a \emph{Hilbert basis} of $L^2(S^1,\mathbb C)$ with respect to $\langle\cdot{,}\cdot\rangle$ and the induced norm $\|f\|_{L^2}:=\sqrt{\langle f,f\rangle}$.

For $z,w\in\C$, we will denote by $[z,w]:=\{z(1-t)+tw,t\in[0,1]\}$ the line segment between $z$ and $w$. Similarly, we denote by $(z,w):=\{z(1-t)+tw,t\in(0,1)\}$ the corresponding line segment without endpoints. Recall that a set $K\subset\C$ is called \emph{convex}, if it contains $[z,w]$ for any two points $z,w\in K$. If its interior contains $(z,w)$ for any two points $z,w\in K$, it is called \emph{strictly convex}. If $\partial K$ admits a parametrization by a closed curve $\Gamma\in C^2(S^1,\C)$ with non-vanishing curvature, then $K$ is called \emph{strongly convex}.

\subsection{Support functions}
Boundaries of strictly convex sets admit parametrizations by so-called \emph{support functions} (see \cite{allemann_equilibria_2021}).

If $K\subset \C$ is strongly convex and $\partial K$ is of class $C^k$, $k\geqslant 2$, then there exists a support function $p\in C^k(S^1)$ such that $\partial K$ is parametrized by
$
t \mapsto p(t) u(t) + p'(t) u'(t),
$
where $u(t)=\mathrm e^{\mathrm it}$ (see \cite[Lemma 2.1]{allemann_equilibria_2021}).

If the boundary $\partial K$ of a strictly convex set $K\subset \C$ is parametrized by $\Gamma:S^1\to\C$,
$\Gamma(t) = p(t)u(t)+p'(t)u'(t)$, where $p$ is of class $C^2$, then we have the following elementary facts \cite[p. 86]{allemann_equilibria_2021}:
\begin{enumerate}
\item
The radius of curvature is given by $\rho=p+p''$. Note that this is not trivial since a priori, $\Gamma$ is only of class $C^1$, see  \cite[Corollary 2.2]{allemann_equilibria_2021} for a proof of this result.\\
\item \label{item2}
If $\rho\geqslant 0$ and $p(t+\pi)+p(t) \equiv d>0$, then $K$ is a body of constant width $d$.\\
\item The area of $K$ is given by
$
\displaystyle\frac{1}{2}\int_0^{2\pi}\!\! \left(p^2(t)-p'(t)^2\right)\,\mathrm dt.\\
$
\item The perimeter of $K$ (the length of $\partial K$) is given by
$\displaystyle
\int_0^{2\pi}\!\!\!\!p(t)\,\mathrm dt.
$
\end{enumerate}

\subsection{Triangles of constant width}\label{sec:2.2}
Let $p(t) = 1+a \cos(3t)$, where $a\geqslant 0$ and $$\gamma(t) = p(t)u(t) + p'(t)u'(t)\textcolor{black}{ = (1+a\cos(3t)+3\mathrm{i} a \cos(3t))\mathrm e^{\mathrm{i}t}}.$$ For $\gamma$ to bound a strictly convex set, we need its curvature radius $\rho$ to be non-negative while being zero only in isolated points. The condition on non-negativity is given by
$$
\rho(t) = p(t) + p''(t) =1-8a\cos(3t)  \geqslant 0,
$$
which implies $a\leqslant \frac18$. Note that if $a<\frac18$, then $\gamma$ bounds a strongly convex set and if $a=\frac18$, then $\rho$ has zeros in $0,\frac{2\pi}{3}$ and $\frac{4\pi}{3}$ within $[0,2\pi)$ so that $\gamma$ bounds a strictly convex set in this case.

Since for all choices $a\in[0,\frac18]$ we have $p(t) +p(t+\pi)\equiv 2$, the curve $\gamma$ bounds a body of constant width $2$ (see elementary fact (\ref{item2})).

We will later focus on the specific choice $a=\frac18$ for the following reason: The \emph{isoperimetric ratio} \textcolor{black}{of a simple closed plane curve with length $\ell$ that encloses the area $A$ is defined as $\ell^2/A$. This quantity is minimized for circles, where it takes the value $4\pi$ according to the isoperimetric inequality. It therefore quantifies how much a closed simple curve fails to be a circle. For the
convex body which is bounded by $\gamma$, it equals}
$$
\frac{4\pi^2}{\pi(1-4a^2)} = \frac{4\pi}{1-4a^2},
$$
which is maximal if $a=\frac18$. In that sense, among all the convex bodies obtained in this way, the case $a=\frac18$ provides us with the one which is farthest from a circle with respect to its isoperimetric ratio. 

\subsection{An isomorphism of $L^2(S^1,\C)$ based on triangles of constant width}

We introduce the notation $f_k(t) = f(kt)$ for any $f\in L^2(S^1,\C)$ and $k\in\mathbb Z$. Then the Fourier-series of $\gamma_k$ is given by \textcolor{black}{(recall that $u(t) = \mathrm e^{\mathrm it}$)}
$$
\gamma_k=u_k+2au_{-2k}-au_{4k},
$$
which provides us with a bounded linear operator $\mathcal T:L^2(S^1,\C)\to L^2(S^1,\C)$ given by
\begin{equation}\label{transform}
\mathcal T(f) = f+2a f_{-2}-af_4.
\end{equation}
Indeed, if $f\in L^2(S^1,\C)$ then $\|\mathcal T(f)\|_{L^2}\leqslant (1+3a)\|f\|_{L^2}$.

\textcolor{black}{As we show in Theorem~\ref{inversetheorem}, $\mathcal T$ admits an inverse, provided $0\leqslant a <\frac13$, which is given by 
\[
\mathcal R(y)(t) = \sum_{\ell = 0}^\infty(-1)^\ell\frac{\left(a+\sqrt{a^2+a}\right)^{\ell+1}-\left(a-\sqrt{a^2+a}\right)^{\ell+1}}{2\sqrt{a^2+a}}y\left((-1)^\ell 2^\ell t\right)
\]
if $0<a<\frac13$ and by $\mathcal R = \mathrm{id}_{L^2(S^1,\mathbb C)}$ if $a=0$. Furthermore, as follows from Theorem~\ref{C^k-restriction}, it can be shown that $\mathcal R$ restricts to a map $C^k(S^1,\mathbb C)\to C^k(S^1,\mathbb C)$, provided $a<{1}/{(2^k(2^k+2))}$ and that this regularity is sharp in the sense that if $a={1}/{(2^k(2^k+2))}$, then the image under $\mathcal R$ of a $C^k$ function is not $C^k$ in general (see Proposition~\ref{regularitysharp}) and the remark thereafter.}

\textcolor{black}{This regularity phenomenon is depicted in Figure \ref{rough24}, where the graph shows a function which is $C^1$ but not $C^2$ and in Figures \ref{rough8} and \ref{rough5}, where the graphs show continuous functions which fail to be differentiable.}

\begin{figure}[H]
\begin{center}
\includegraphics[scale=0.75]{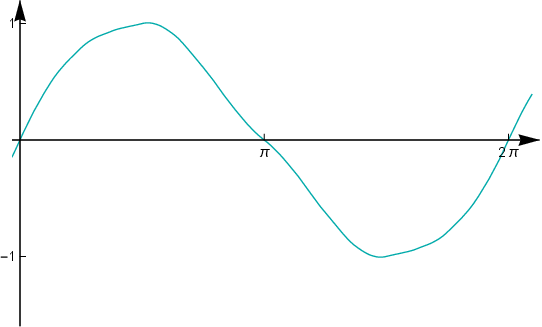}\quad\includegraphics[scale=0.75]{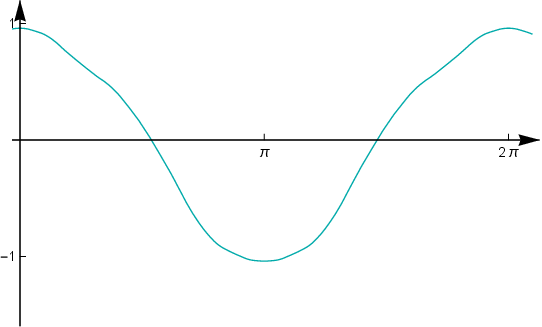}
\caption{The graphs of $\mathcal R(\sin)$ and $\mathcal R(\cos)$ as real-valued functions if $a=\frac{1}{24}$.\label{rough24}}
\end{center}
\end{figure}
\begin{figure}[H]
\begin{center}
\includegraphics[scale=0.75]{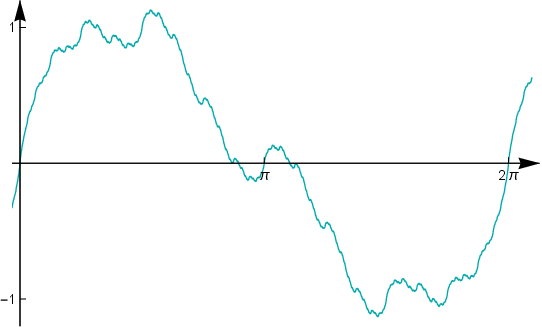}\quad\includegraphics[scale=0.75]{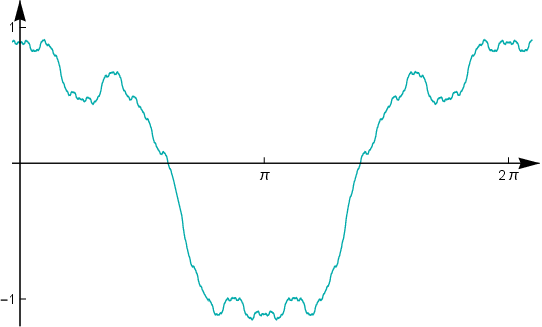}
\caption{The graphs of $\mathcal R(\sin)$ and $\mathcal R(\cos)$ as real-valued functions if $a=\frac18$.\label{rough8}}
\end{center}
\end{figure}

\begin{figure}[H]
\begin{center}
\includegraphics[scale=0.75]{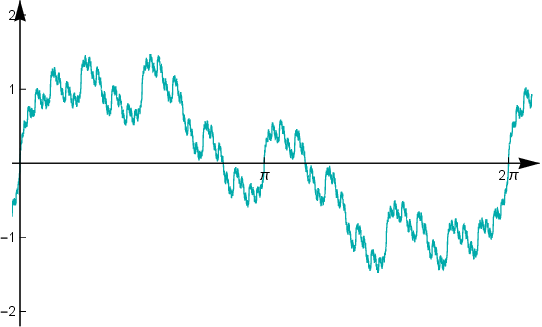}\quad\includegraphics[scale=0.75]{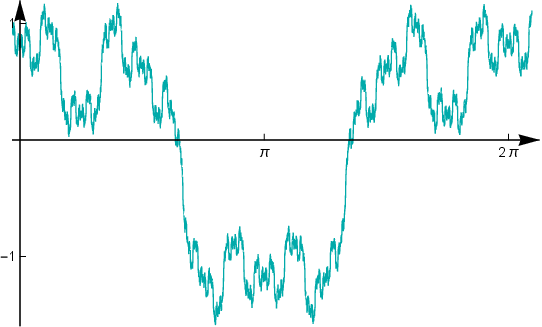}
\caption{The graphs of $\mathcal R(\sin)$ and $\mathcal R(\cos)$ as real-valued functions if $a=\frac15$.\label{rough5}}
\end{center}
\end{figure}
Note that by construction, if $a\in[0,\frac13)$, then $\mathcal R(\gamma_k)=u_k$ so that the we arrive at:
\begin{proposition}\label{fourier3}
If $a\in[0,\frac13)$, then set $\mathcal B = \{\gamma_k\}_{k\in\mathbb Z}$ is a Hilbert basis for the space $L^2(S^1,\mathbb C)$ equipped with the inner product $\mathcal R^*\langle\cdot{,}\cdot\rangle = \Langle\cdot{,}\cdot\Rangle$.
\end{proposition}
The inner product $\Langle\cdot{,}\cdot\Rangle$ is explicitly given by
$$
\Langle f,g\Rangle = \frac{1}{2\pi}\int_0^{2\pi} \mathcal R(f)(t)\overline{\mathcal R(g)(t)}\,\mathrm dt
$$
so that a periodic function $f\in L^2(S^1,\C)$ can be represented by
$$
\sum_{k\in\mathbb Z}\Langle f,\gamma_k\Rangle \gamma_k.
$$
Since
$$\begin{aligned}
\Langle f,\gamma_k\Rangle & = \frac{1}{2\pi}\int_0^{2\pi}\mathcal R(f)(t)\,\overline{\mathcal R(\gamma_k)(t)}\,\mathrm dt\\
& = \frac{1}{2\pi}\int_0^{2\pi}\mathcal R(f)(t)\,\overline{u_k(t)}\,\mathrm dt,
\end{aligned}
$$
the coefficients $\Langle f,\gamma_k\Rangle$ are precisely the Fourier-coefficients of $\mathcal R(f)$.

This can be restated as follows: If one defines

$$\begin{aligned}
\mathcal C_k(\mathcal F): L^2(S^1,\C) & \to \C\\
f & \mapsto \langle f,u_k\rangle
\\
\mathcal C_k(\mathcal R): L^2(S^1,\C) & \to \C\\
f & \mapsto \Langle f,\gamma_k\Rangle,
\end{aligned}$$
for all $k\in\Z$, then the following diagram commutes:

%

$$\xymatrix{
  L^2(S^1,\C) \ar[dr]_{\mathcal C_k(\mathcal R)} \ar[rr]^{\mathcal{R}} & & L^2(S^1,\C) \ar[dl]^{\mathcal C_k(\mathcal F)} \\
 & \C& }
$$

\begin{remark}
The map $\gamma$ yields a parametrization of a closed curve in $\C$ for any value of $a$. 
It follows from Proposition \ref{fourier3}, that a Fourier-type series with the unit circle replaced by $\gamma$ is available if $a<\frac13$. We therefore also consider curves $\gamma$, where $a\in\left(\frac18,\frac13\right)$. For $a$ in this range, $\gamma$ is no longer the boundary of a convex body but a curve with self-intersections. The image of $\gamma$ whenever $a=\frac15$ is shown in \textcolor{black}{Figure} \ref{trianglewithselfintersection}.
\begin{figure}[H]
\begin{center}
\includegraphics[scale=0.65]{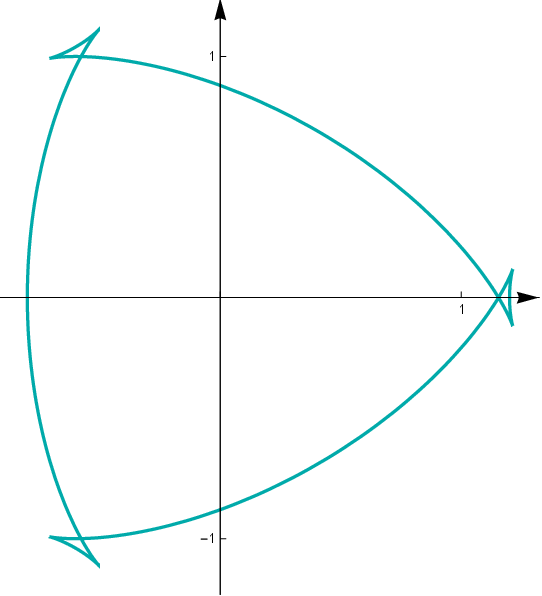}
\caption{Image of $\gamma$ if $a=\frac15$}
\label{trianglewithselfintersection}
\end{center}
\end{figure}
\end{remark}

\section{Application to letterforms}\label{type-design}

As long as $a\in\left[0,\frac13\right)$, we have that $\mathcal T\circ \mathcal R = \operatorname{id}_{L^2(S^1,\mathbb C)}$. We will now introduce a truncated version of $\mathcal R$:
\begin{equation}\begin{aligned}
\mathcal R_n : L^2(S^1,\mathbb C) & \to L^2(S^1,\mathbb C)\\
y \mapsto
\mathcal R_n(y): t & \mapsto \sum_{\ell = 0}^n(-1)^\ell\frac{\left(a+\sqrt{a^2+a}\right)^{\ell+1}-\left(a-\sqrt{a^2+a}\right)^{\ell+1}}{2\sqrt{a^2+a}}y((-1)^\ell 2^\ell  t).
\end{aligned}\label{truncatedbacktransform}
\end{equation}

The way we will use $\mathcal R_n$ in order to modify letterforms of which we think as H\"older-continuous elements $\mathrm L\in L^2(S^1,\C)$ is best summarized in terms of the following commutative diagram:

$$\xymatrix{
  L^2(S^1,\C) \ar[r]^{\mathcal T} \ar[rd]_{\mathrm{id}} & L^2(S^1,\C) \ar[r]^{\mathcal{R}_n} \ar[d]^{\mathcal R} & L^2(S^1,\C) \ar[dl]^{n\to\infty}  \\
& L^2(S^1,\C)& }
$$

In this way, $(\mathcal R_{n}\circ \mathcal T)(\mathrm L)\approx \mathrm L$, where the approximation becomes exact as $n\to\infty$. We obtain thus different approximations as we choose different values for $a\in (0,\frac13)$ and $n\in \N$.

As a case study, we will choose the values $a=\frac{1}{24}$, $a=\frac18$ and $a=\frac15$ for different values of $n$. Note that if $a=\frac{1}{2^k(2^k+2)}$, then the coefficients in the series defining $\mathcal R$ take a particularly nice form, since in this case
$$
\sqrt{a^2+a} = \frac{2^k+1}{2^k(2^k+2)}\in\mathbb Q.
$$
This occurs for $a=\frac{1}{8}$ and $a=\frac{1}{24}$, which correspond to the choices $k=1$ and $k=2$ \textcolor{black}{respectively}.

\textcolor{black}{Before presenting the examples, a few remarks regarding our implementation are in order: We start with a letterform which is given e.g.\ as a bitmap file and extract the contours of the letters using Mathematica, which results in a list of points in the plane representing the contour. From there, a discrete Fourier transform is performed which gives rise to the contour being given as a $2\pi$-periodic map, in fact, a truncated Fourier series with $m$ being very large, so that there is no visual difference between the original letter and its Fourier approximant. From there, it is straightforward to truncate this series or to compute $\mathcal R_n\circ \mathcal T$ using equations \eqref{transform} and \eqref{truncatedbacktransform}.}
\subsection{Approximation with $a=\frac{1}{24}$}

Recall that if $a=\frac{1}{24}$, then $\mathcal R$ restricts to a map $C^1(S^1,\C)\to C^1(S^1,\C)$ which does not in general preserve $C^2$ regularity according to Theorem \ref{C^k-restriction}. Here,
$$
\mathcal R(y)(t) =  \lim_{n\to\infty} \underbrace{\frac15\sum_{\ell=0}^n \frac{(-\frac{1}{2})^\ell 3^{\ell+1}+2}{6^\ell}\,y\left((-1)^\ell 2^\ell t\right)}_{=:\mathcal R_n(y)(t)}.
$$

Appliying $\mathcal R_n\circ\mathcal T$ to the letter \textcolor{black}{``S''} from the Romain 20 font 
yields the pictures in \textcolor{black}{Figure} \ref{1/24}.

\begin{figure}[H]
\begin{center}
\includegraphics[scale=0.36]{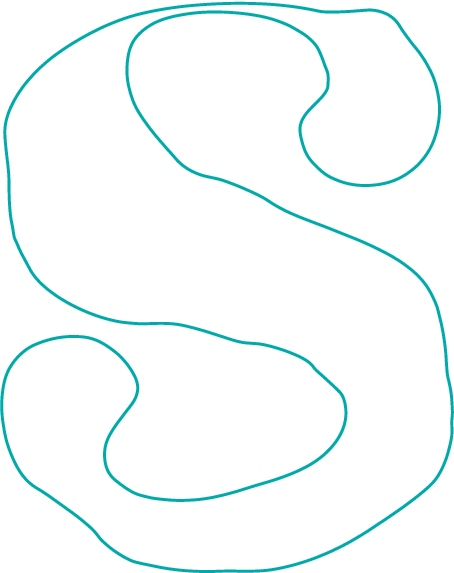}\qquad
\includegraphics[scale=0.36]{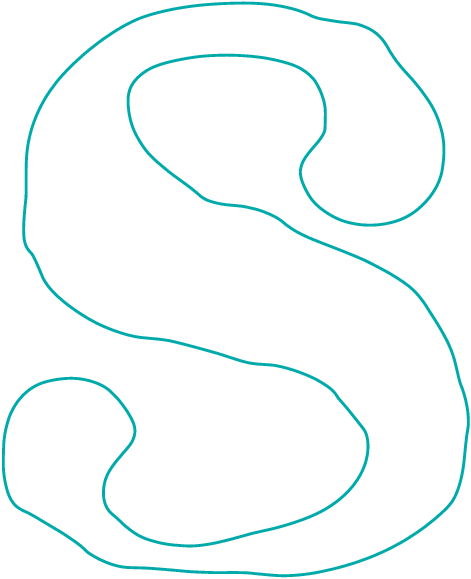}\qquad
\includegraphics[scale=0.36]{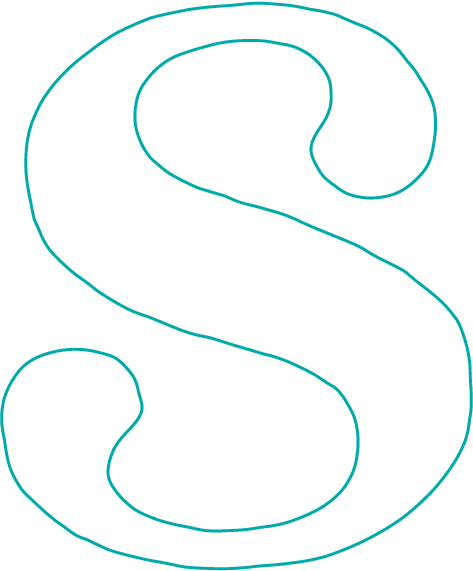}\qquad
\includegraphics[scale=0.36]{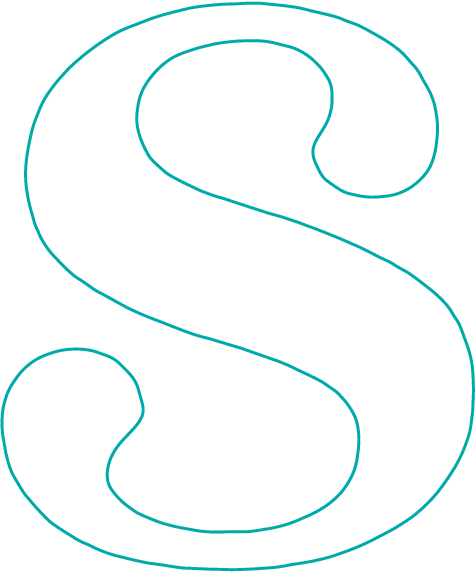}
\caption{Application of $\mathcal R_n\circ\mathcal T$ to the letter \textcolor{black}{``S''} from the Romain 20 font for the values $n=0,1,2,3$ if $a=\frac{1}{24}$}\label{1/24}
\end{center}
\end{figure}

Even though the curve $\gamma$ for $a=\frac{1}{24}$ is relatively close to being a circle (see the left image in \textcolor{black}{Figure} \ref{fig:trianglesofconstantwidth}), the application of $\mathcal R_n\circ\mathcal T$ to the letter \textcolor{black}{``S''} produces noticeably different curves, compare to \textcolor{black}{Figure} \ref{alice-fourier}.

\subsection{Approximation with $a=\frac{1}{8}$}

If $a=\frac{1}{8}$, $\gamma$ corresponds to the triangle of constant width which is furthest from the circle with respect to its isoperimetric ratio (see Subsection \ref{sec:2.2}) and $\mathcal R$ restricts to a map $C^0(S^1,\C)\to C^0(S^1,\C)$ which does not in general preserve $C^1$ regularity. Here,
$$
\mathcal R(y)(t) =  \lim_{n\to\infty} \underbrace{\frac13\sum_{\ell=0}^n \frac{(-1)^\ell 2^{\ell+1}+1}{4^\ell}\,y\left((-1)^\ell 2^\ell t\right)}_{=:\mathcal R_n(y)(t)}.
$$

\begin{figure}[H]
\begin{center}
\includegraphics[scale=0.45]{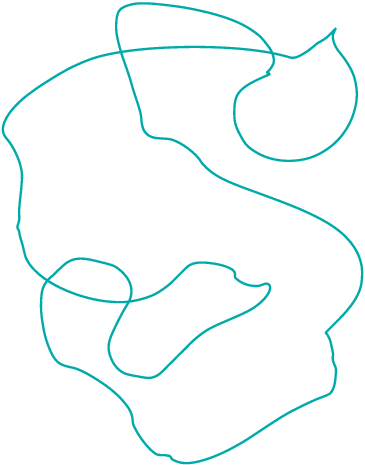}\qquad
\includegraphics[scale=0.45]{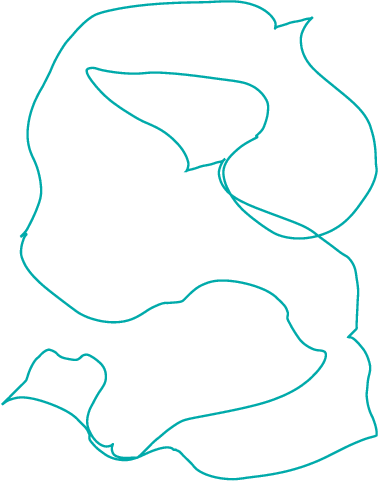}\qquad
\includegraphics[scale=0.45]{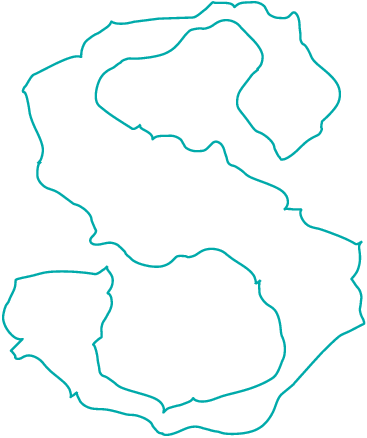}\qquad
\includegraphics[scale=0.45]{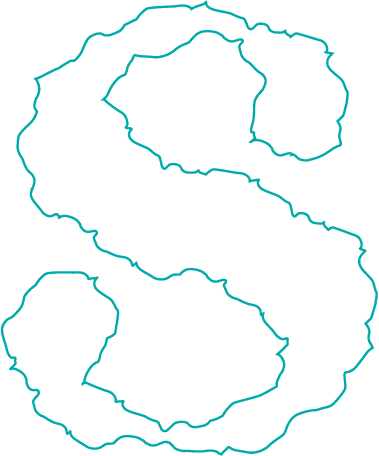}\\[0.5cm]
\includegraphics[scale=0.45]{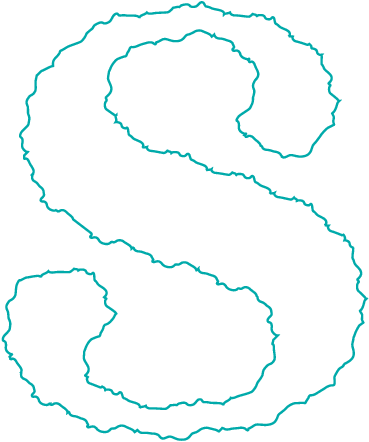}\qquad
\includegraphics[scale=0.45]{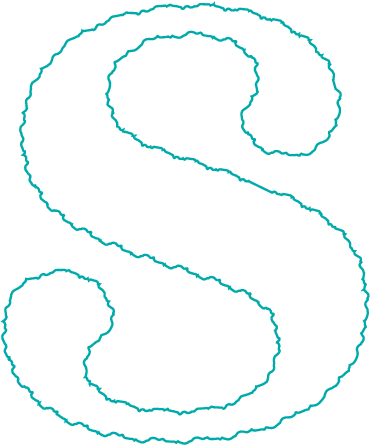}\qquad
\includegraphics[scale=0.45]{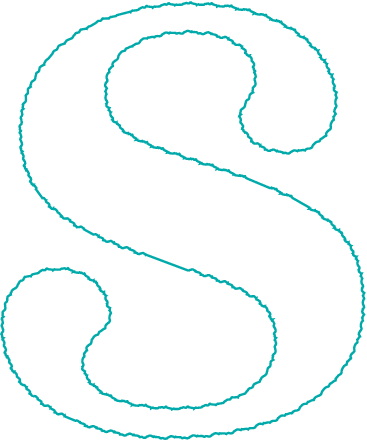}
\caption{Application of $\mathcal R_n\circ\mathcal T$ to the letter \textcolor{black}{``S''} from the Romain 20 font for the values $n=0,\ldots,6$ if $a=\frac18$}\label{1/8}
\end{center}
\end{figure}

\textcolor{black}{The resulting letterforms are shown in Figure \ref{1/8}}. It is an artefact of the regularity that the curves have angular points which give the letterforms a certain roughness like a vibrating, fuzzy object -- an effect which would barely be obtainable by the use of B\'ezier curves.

\subsection{Approximation with $a=\frac{1}{5}$}
Recall that if $\frac18 <a<\frac13$, then the image of $\gamma$ is no longer the boundary of a convex body, but $\gamma$ has self-intersections (see Figure~\ref{trianglewithselfintersection}). Nonetheless, $\mathcal R$ is still an isomorphism and preserves the continuity of mappings. If $a=\frac15$, one obtains the images given in \textcolor{black}{Figure} \ref{1/5}, when applying $\mathcal R_n\circ\mathcal T$ to the letter \textcolor{black}{``S''}. Here, the curves admit self-intersections and this produces a fractal-like decorative effect.

\begin{figure}[H]
\begin{center}
\includegraphics[scale=0.46]{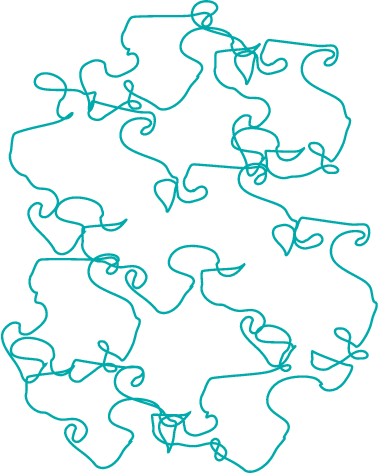}\qquad
\includegraphics[scale=0.46]{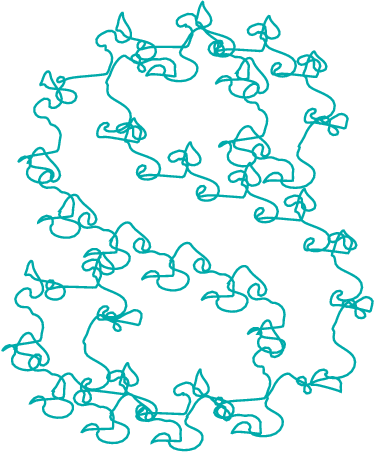}\qquad
\includegraphics[scale=0.46]{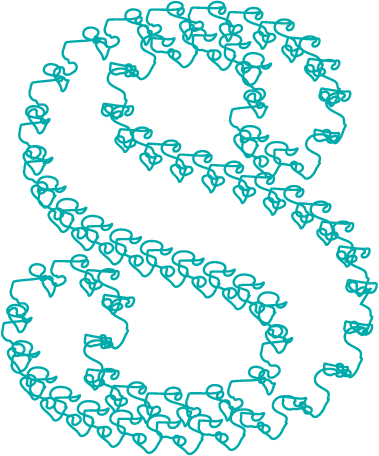}\qquad
\includegraphics[scale=0.46]{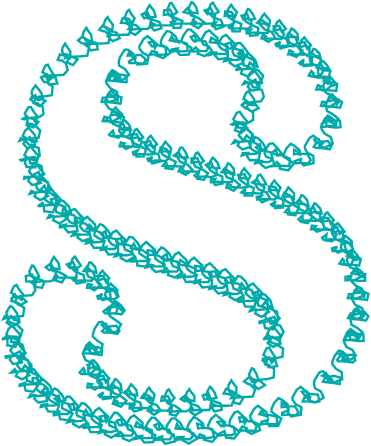}
\caption{Application of $\mathcal R_n\circ\mathcal T$ to the letter \textcolor{black}{``S''} from the Romain 20 font for the values $n=3,4,5,6$ if $a=\frac15$}\label{1/5}
\end{center}
\end{figure}

\subsection{\textcolor{black}{More realistic use case}}

\textcolor{black}{Until now, we have focused on modifying the contour of a letterform. In reality, however, a font is a filled character and generally small in size.} 
\textcolor{black}{Table \ref{table-font-trans} shows the letters ``HCFG'' from Mitim Sigma together with the letter ``S'' from the Romain 20 font in their original and modified versions as filled contours in order to show how the letterforms look like in reduced size.}

\begin{table}[H]
\begin{center}
\sffamily
\renewcommand{\arraystretch}{2.2}
\textcolor{black}{
\begin{tabular}{|>{\raggedleft\arraybackslash}m{5cm}|>{\centering\arraybackslash}b{2.5cm}|>{\centering\arraybackslash}b{2.1cm}|}\hline
& Mitim Sigma & Romain 20\\
\hline
Original letters & \includegraphics[scale=0.17]{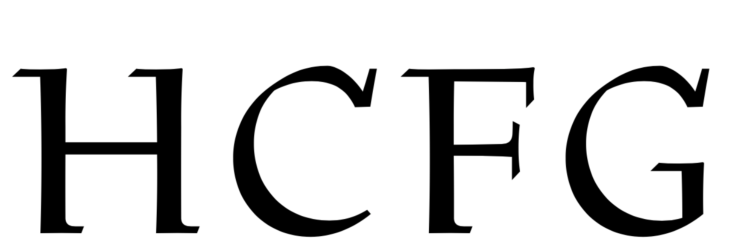}&\includegraphics[scale=0.1]{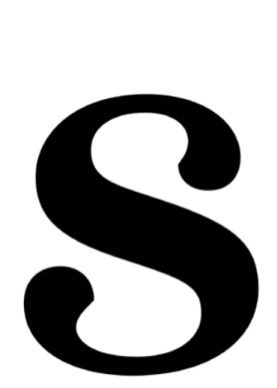}\\ \hline
Truncated Fourier with $m=10$ & \includegraphics[scale=0.06]{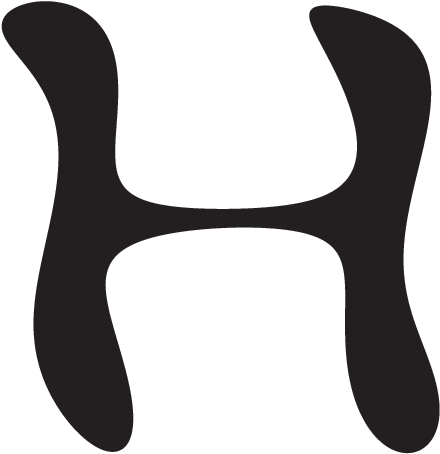}
\includegraphics[scale=0.05]{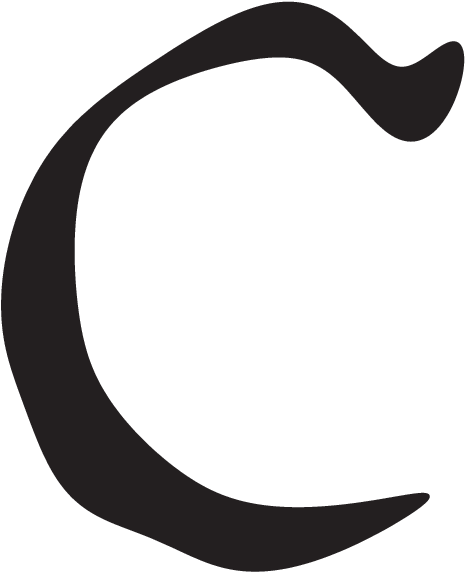}
\includegraphics[scale=0.05]{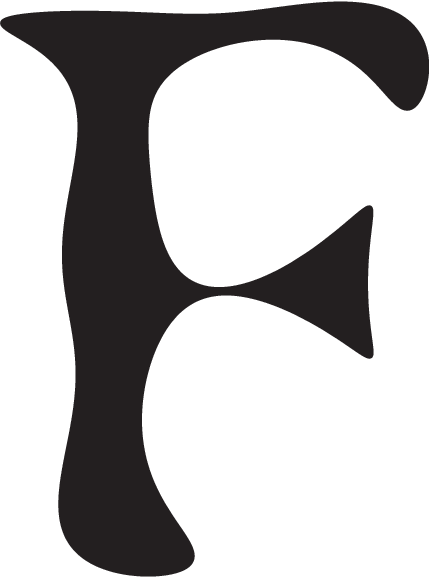}
\includegraphics[scale=0.05]{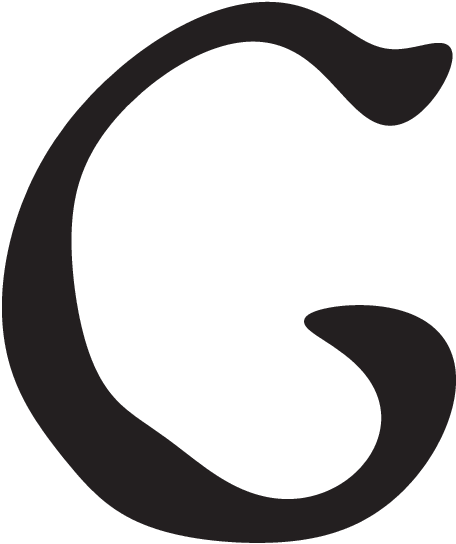}&
\includegraphics[scale=0.05]{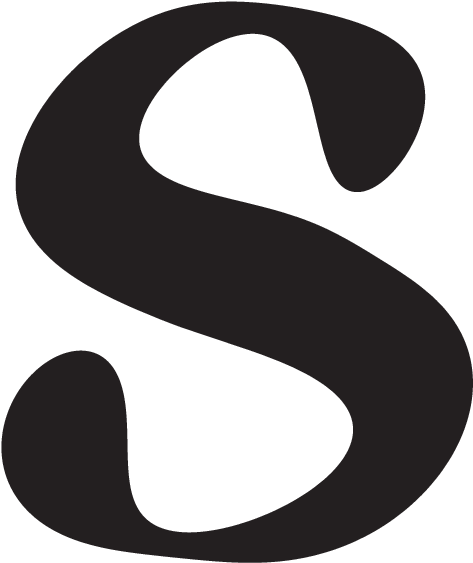}\\ \hline
$\mathcal R_2\circ \mathcal T$ with $a=\frac{1}{24}$&
\includegraphics[scale=0.065]{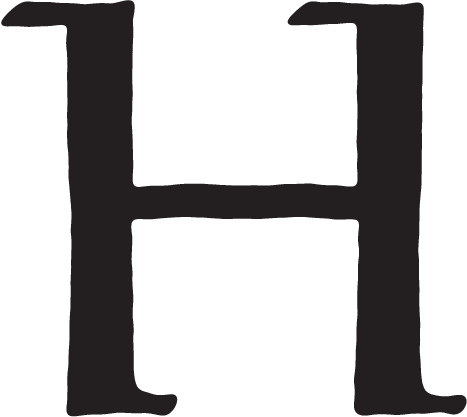}
\includegraphics[scale=0.05]{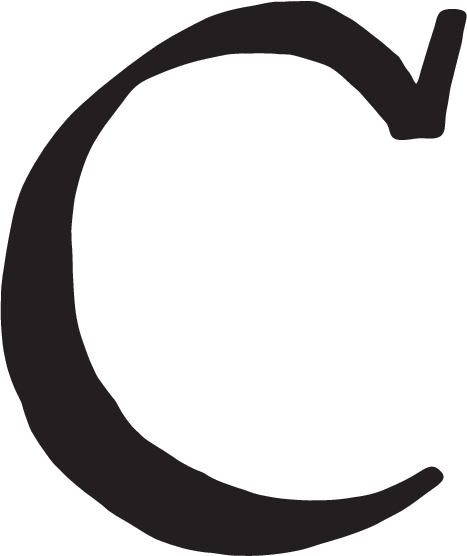}
\includegraphics[scale=0.05]{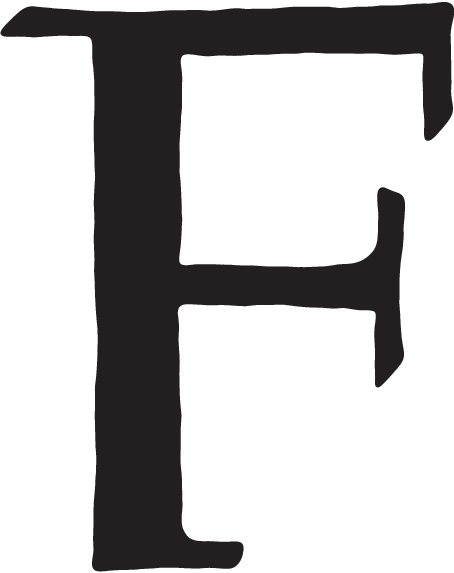}
\includegraphics[scale=0.05]{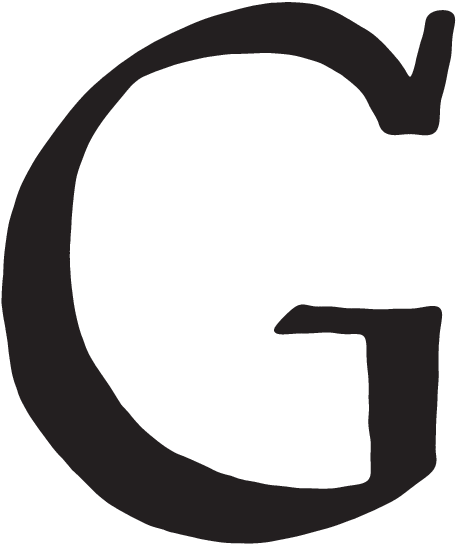}&
\includegraphics[scale=0.05]{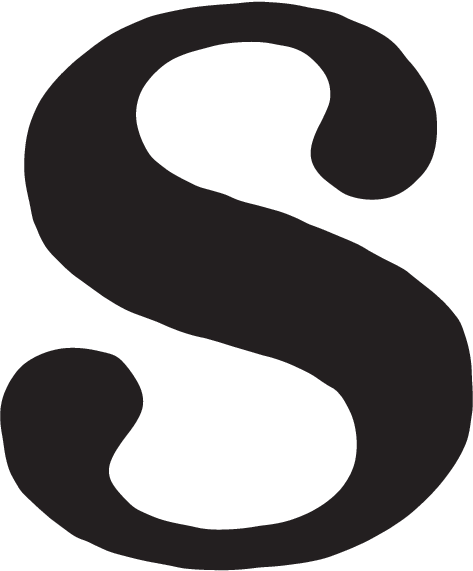}\\ \hline
$\mathcal R_2\circ \mathcal T $ with $a=\frac{1}{8}$&
\includegraphics[scale=0.065]{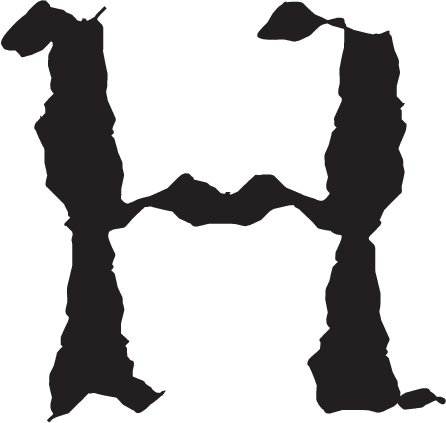}
\includegraphics[scale=0.05]{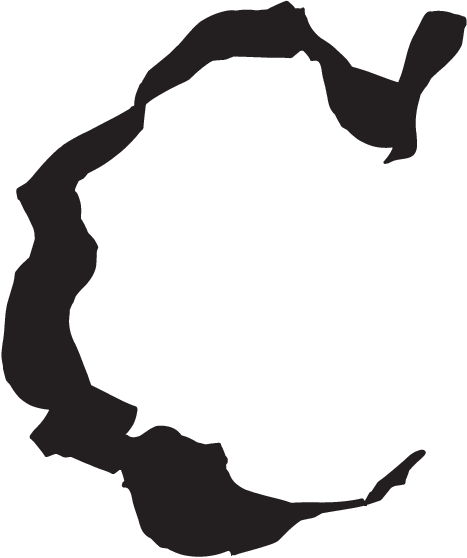}
\includegraphics[scale=0.05]{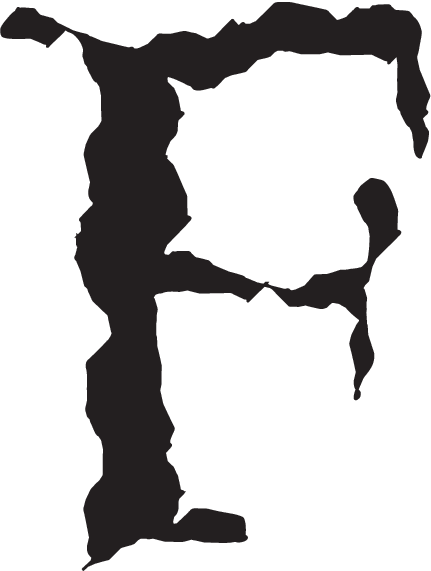}
\includegraphics[scale=0.05]{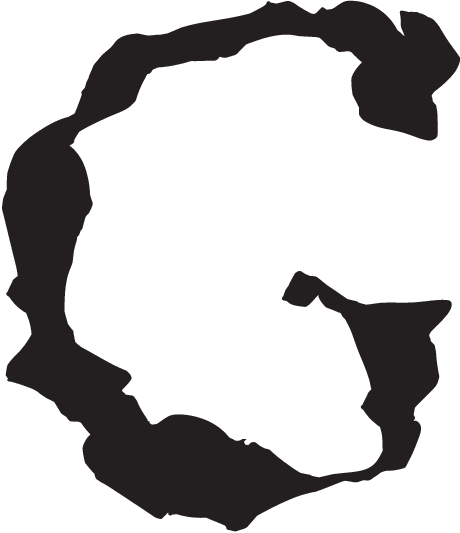}&
\includegraphics[scale=0.05]{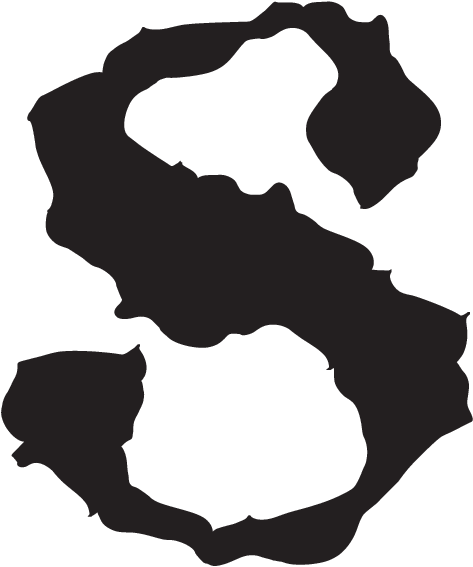}\\
\hline
\end{tabular}}\\[0.4cm]
\caption{\textcolor{black}{Transformations applied to some black-filled letters as they appear in practice.\label{table-font-trans}}}
\end{center}
\end{table}

\section{Conclusion \textcolor{black}{and perspectives}}

The approach to the creation of letterforms presented in this paper shows an alternative to the use of Bézier curves. Although we do not claim that the letters we have created in this way can be used in an unmodified way nor how such an approach could be incorporated in a type design software, it shows that a functional-analytical approach, which changes letters globally and not locally, takes into account the common criticism that Bézier letters tend to be too clean. For our choice of triangles of constant width, the resulting map $\mathcal R$ reveals non-trivial regularity phenomena, which allow\textcolor{black}{s} \textcolor{black}{us} to create letterforms with self-intersections, or letters 	\textcolor{black}{that} do not have tangents in general, creating, wobbly, \textcolor{black}{ornate, erratic or even decorative} effects. \textcolor{black}{These effects are explained quantitatively in Theorem~\ref{C^k-restriction} so that the choice of $a$ (and hence $k$) gives a way to control the visual smoothness of the obtained letters.}

\textcolor{black}{In our transformation, the constant-width triangles are all aligned in the same orientation. By introducing a parameter to control the phase of the curve, we could explore how this influences the geometry of the letterforms. This additional degree of freedom would provide new possibilities for variation, enriching the artistic potential of the design.}

In addition, the approach provides a wealth of exploration possibilities, because \textcolor{black}{basis functions other than} the ones we have chosen, are imaginable -- for example, non-differentiable curves such as regular $n$-gons might be used as a geometric basis object.

\section*{Acknowledgements}

We would like to sincerely thank Raphaela Häfliger, Alice Savoie, Kai Bernau, Nicolas Bernklau, Matthieu Cortat, Roland Früh and Radim Pe\v{s}ko from ECAL for their interest in this work as well as for their valuable comments. Also, we would like to specifically thank Alice Savoie for kindly making her Romain 20 letter ``S'' available to us and Radim Pe\v sko for providing letterforms of his font Mitim Sigma. Furthermore, we would like to thank the two anonymous referees for their important and valuable comments and suggestions, which have helped to improve the quality of this article. Also, we thank HEIA Fribourg for covering the Open Access Publication fee and the HC Fribourg-Gottéron for the inspiring choice of the letterforms HCFG. 

\section*{Funding}
This study has been conducted as part of the interdisciplinary research project ``Beyond B\'ezier -- Exploration of Drawing Methods in Type Design'' where the primary goal consisted in finding alternative approaches to the creation of letterforms other than B\'ezier curves. This project and hence this study were financed by the ``Design et arts visuels'' domain of HES-SO, University of Applied Sciences Western Switzerland.

\newpage
\appendix
\section{Mathematical proofs}

\begin{lemma}\label{lemmajacobsthal}
If $a>0$, the sequence $(j_n)_{n\in\mathbb N}$ with $j_0=1, j_1 = 2a$ verifying $j_n = a(2j_{n-1}+j_{n-2})$ for $n\geqslant 2$ admits an explicit representation
$$
j_n = \frac{\left(a+\sqrt{a^2+a}\right)^{n+1}-\left(a-\sqrt{a^2+a}\right)^{n+1}}{2\sqrt{a^2+a}}.
$$
\end{lemma}

\begin{proof}
The real vector space $V$ of sequences $(z_n)_{n\in\mathbb N}$ that solve ${z_n = a(2z_{n-1}+z_{n-2})}$ for all $n\geqslant 2$ admits a basis consisting of exponential functions: Let $
z_n := b^n$, where $b\ne 0$. The recurrence relation implies
$$
b^2 = a(2b+1) \Longleftrightarrow b = a \pm \sqrt{a^2+a}.
$$
Letting $z^\pm_n =(a \pm \sqrt{a^2+a})^n$, it is straightforward to check that $\{(z^+_n)_{n\in\N},(z^-_n)_{n\in\N}\}$ forms a basis of $V$. Therefore, $j_n$ will be a linear combination of $(a + \sqrt{a^2+a})^n$ and $(a - \sqrt{a^2+a})^n$ and using $j_0=1$, $j_1=a$, we obtain the desired result.
\end{proof}

\begin{theorem}\label{inversetheorem}
Let $a\in[0,\frac13)$. Then $\mathcal T:L^2(S^1,\C)\to L^2(S^1,\C)$ defined by $\mathcal T(f) = f+2a f_{-2}-af_4$ is an isomorphism and its inverse $\mathcal R: L^2(S^1,\C)\to L^2(S^1,\C)$ is given by
$$
\mathcal R(y)(t) = \sum_{\ell = 0}^\infty(-1)^\ell\frac{\left(a+\sqrt{a^2+a}\right)^{\ell+1}-\left(a-\sqrt{a^2+a}\right)^{\ell+1}}{2\sqrt{a^2+a}}y\left((-1)^\ell 2^\ell t\right),
$$
if $a\ne 0$ and by the identity mapping on $L^2(S^1,\C)$ if $a=0$. Moreover, $\mathcal R$ restricts to a map $C^0(S^1,\mathbb C)\to C^0(S^1,\mathbb C)$.
\end{theorem}

\begin{proof}
If $a=0$, then $\mathcal T=\mathrm{id}_{L^2(S^1,\C)}$ and there is nothing to show. Let therefore $a>0$. If $\mathcal T(f) = y$ it holds that
\begin{align}\label{representf}
f & = y-2af_{-2}+af_4.
\end{align}
Since
\begin{equation}\label{startitequation}
f_{k}  = y_{k}-af_{-2k}+2af_{4k}
\end{equation}
for all $k\in\mathbb Z$ we can, using \eqref{startitequation}, replace $f_{-2}$ in \eqref{representf} and obtain
\begin{equation}\label{firstiteration}
f = y- 2ay_{-2}+2a^2f_{-4}-2a^2f_8.
\end{equation}
Next, $f_{-4}$ in \eqref{firstiteration} can be replaced using \eqref{startitequation} so that proceeding inductively in this way, $f$ can eventually be expressed as a formal series
\begin{equation}\label{formalseries}
f(t) = \sum_{\ell=0}^\infty (-1)^\ell j_\ell\, y\left((-1)^\ell 2^\ell t\right),
\end{equation}
where the coefficients $j_\ell$ can be checked to verify the following recurrence relation: $j_0 = 1, j_1 = 2a$ and
$j_n = a(2j_{n-1}+j_{n-2})\text{ for }n\geqslant 2.$
According to Lemma~\ref{lemmajacobsthal}, it holds that
$$
j_n = \frac{\left(a+\sqrt{a^2+a}\right)^{n+1}-\left(a-\sqrt{a^2+a}\right)^{n+1}}{2\sqrt{a^2+a}},
$$
so that according to \eqref{formalseries}, the inverse of $\mathcal T$ is given by the formal series $\mathcal T^{-1}=\mathcal R$:
$$
\mathcal R(y)(t) = \sum_{\ell=0}^\infty(-1)^\ell j_\ell \,y\left((-1)^\ell 2^\ell t\right).
$$
If $y\in L^2(S^1,\mathbb C)$, then $
\left\|\mathcal R(y)\right\|_{L^2}\leqslant\|y\|_{L^2} \sum_{\ell=0}^\infty j_\ell 
$
so that this expression is finite if the series
$
\sum_{\ell=0}^\infty j_\ell
$
converges. Since
$$
\lim_{n\to\infty}\frac{j_{n+1}}{j_n}=a+\sqrt{a^2+a},
$$
this series converges according to the ratio test if $a<\frac13$. Replacing the $L^2$-norm in the above computation by the $C^0$-norm, it is immediate that $\mathcal R$ restricts to a map $C^0(S^1,\mathbb C)\to C^0(S^1,\mathbb C)$ provided $a<\frac13$.
\end{proof}

\begin{theorem}\label{C^k-restriction}
Let $a\in[0,\frac13)$. If $a<\frac{1}{2^k(2^k+2)}$, then $\mathcal R$ restricts to a map $C^k(S^1,\mathbb C)\to C^k(S^1,\mathbb C)$.
\end{theorem}

\begin{proof}
Suppose that $y\in C^k(S^1,\mathbb C)$, then the $k$-th derivative (taken term-wise) of $\mathcal R(y)$ is formally given by
\begin{equation}\label{formalkderivative}
\sum_{\ell=0}^\infty (-1)^{(k+1)\ell} 2^{k \ell}j_\ell y^{(k)}((-1)^\ell 2^\ell \, \cdot \,)
\end{equation}
and its $C^0$ norm will be bounded by
$
\|y^{(k)}\|_{C^0}\sum_{\ell = 0}^\infty 2^{k\ell}j_\ell.
$
This time, the ratio test yields
$$
\lim_{n\to\infty}\frac{2^{k(n+1)}j_{n+1}}{2^{kn}j_n}=2^k\left(a+\sqrt{a^2+a}\right),
$$
which is smaller than $1$ provided $a<\frac{1}{2^k(2^k+2)}$, so that in this case the series \eqref{formalkderivative} converges uniformly and hence $\mathcal R$ restricts to $C^k(S^1,\mathbb C)\to C^k(S^1,\mathbb C)$.
\end{proof}

\begin{proposition}\label{regularitysharp}
If $y(t) = \sin(t)$ and $a=\frac18$, then $ f=\mathcal R(y):S^1\to\C$ is not differentiable in $t=0$.
\end{proposition}
\begin{proof}
Let $a=\frac18$ and $y(t) = \sin(t)$, then
$$
f(t)=\mathcal{R}(y)(t)=\frac13\sum_{\ell=0}^\infty \frac{ 2^{\ell+1}+(-1)^\ell}{4^\ell}\sin\left(2^\ell t\right)\color{black}{.}
$$
We will show that if $t_n = \frac{\pi}{2^n}$, then the sequence
$$
\frac{f(t_n)-f(0)}{t_n-0} = \frac{f(t_n)}{t_n}
$$
will diverge as $n\to\infty$.
We will use the inequality
$$\frac{2^{\ell+1}+(-1)^\ell}{4^\ell} \geqslant \frac{1}{2^\ell}$$
for all $\ell\in\mathbb N$. Since $\sin(2^\ell t_n) = \sin\left(2^{\ell-n}\pi\right)$, this quantity vanishes if $\ell\geqslant n$. If $\ell < n$ we can use the inequality $\sin(x)\geqslant \frac2\pi x$ that holds on $[0,\frac\pi2]$ in order to obtain
$$
\begin{aligned}
f(t_n) & \geqslant \frac13\sum_{\ell=0}^{n-1} \frac{1}{2^\ell}\frac{2}{\pi} 2^{\ell-n}\pi=\frac{2n}{3\cdot 2^n}.
\end{aligned}$$
Therefore we have
$$
\frac{f(t_n)}{t_n}\geqslant \frac{2^n}{\pi}\cdot \frac{2n}{3\cdot 2^n} = \frac{2n}{3\pi}\stackrel{n\to\infty}{\longrightarrow}\infty.
$$
\end{proof}

\begin{remark}
The previous proposition shows that the image under $\mathcal R$, if $a=\frac18$, of a $C^1$ function $S^1\to \C$ is not $C^1$ in general and the result can be adapted in order to show that if $a={1}/{(2^k(2^k+2))}$, then the image under $\mathcal R$ of a $C^k$ function is not $C^k$ in general.
\end{remark}

\newpage
\newcommand{\etalchar}[1]{$^{#1}$}

\end{document}